\documentclass[11pt]{amsart}
\usepackage{
amssymb,
amsmath}

\usepackage[bookmarks,colorlinks,pagebackref]{hyperref} 

\usepackage{graphicx}

\usepackage{hyperref}

\usepackage[latin1]{inputenc}

\usepackage{mathpazo}
\usepackage[scaled=.95]{helvet}
\usepackage{courier}

\usepackage[all,cmtip,2cell]{xy}

\numberwithin{equation}{section}
\theoremstyle{plain} 
\newtheorem{proposition}{Proposition}[section]  
\newtheorem{lemma}[proposition]{Lemma}
\newtheorem{corollary}[proposition]{Corollary} 
\newtheorem{theorem}[proposition]{Theorem} 

\theoremstyle{definition} 
\newtheorem{definition}[proposition]{Definition}

\newtheorem{defob}[proposition]{Definition and observation}

\newtheorem{example}[proposition]{Example} 

\newtheorem{question}[proposition]{Question}

\newtheorem{notation}[proposition]{Notation}

\newtheorem{construction}[proposition]{Construction}
\newtheorem{notation and recalls}[proposition]{Notations and Recalls}

\newcommand\Hom{\operatorname{Hom}}

\newcommand\Rad{\operatorname{Rad}}
\newcommand\Ker{\operatorname{Ker}}

\newcommand{\id}{\operatorname{id}}

\newcommand{\xx}{\underline x}

\newcommand\Spec{\operatorname{Spec}}

\newcommand{\Cech}{{\Check {C}}_{\xx}}

\newcommand{\la}{\langle}
\newcommand{\ra}{\rangle}

\author[P.~Schenzel]{Peter Schenzel}
\title[Deligne's formula]{A note on Deligne's formula}
	
\address{Martin-Luther-Universit\"at Halle-Wittenberg,
Institut f\"ur Informatik, D --- 06 099 Halle (Saale), Germany}
\email{peter.schenzel@informatik.uni-halle.de}

\date{\today}

\begin{document}

\begin{abstract} 
	Let $R$ denote a 	Noetherian ring and an ideal $J \subset R$ with 
	$U = \Spec R \setminus V(J)$. For an $R$-module $M$ there is an 
	isomorphism $\Gamma(U, \tilde{M}) \cong \varinjlim \Hom_R(J^n,M)$ 
	known as Deligne's formula (see \cite[p. 217]{Hr5} and Deligne's Appendix 
	in \cite{Hr1}). We extend the 
	isomorphism for any $R$-module 
	$M$ in the non-Noetherian case of $R$ and $J = (x_1,\ldots,x_k)$ a 
	certain finitely generated ideal. Moreover, we  recall a corresponding sheaf construction.
\end{abstract}


\maketitle

\section{Introduction}
Let $R$ denote a commutative Noetherian ring, $J \subseteq R$ an ideal 
and $U = \Spec R \setminus V(J)$. For an $R$-module $M$ and its associated 
sheaf $\tilde{M}$ on $X = \Spec R$ it is known that the sheaf cohomology $H^i(U,\tilde{M})$ 
and the \v{C}ech cohomology $\check{H}^i(U,\tilde{M})$ are isomorphic for all $i 
\in \mathbb{Z}$ (see e.g. \cite[III, 4]{Hr5}). In the non-Noetherian case a corresponding  result 
holds whenever $J$ is generated by a weakly pro-regular sequence $\xx = x_1,\ldots,x_k$ 
(for the definition see \cite{SS}) and a covering of $U$ by $\Spec R \setminus V(x_i), 
i = 1,\ldots,k$. For the details we refer to \cite{Sp2} and the monograph \cite{SS}, where 
it is worked out in the frame of Commutative Algebra.

In the case of a Noetherian ring it is well-known that the global transform 
$\mathcal{D}_J(M) \cong  \varinjlim \Hom_R(J^n,M) $ is isomorphic 
to $\check{H}^0(U,\tilde{M}) \cong \Gamma( U,  \tilde{M}),  U = X \setminus V(J),$ known as Deligne's formula 
(see e.g. \cite{Hr1}, \cite{Hr5}, \cite{BrS} and \cite{StV}). 
Moreover,  we shall contribute with  a variant of Deligne's formula in the non-Noetherian case 
(generalizing arguments of \cite{Hr1} and \cite{StV}) for some particular classes of ideals 
$J$ generated by $\xx = x_1,\ldots, x_k$, where we put 
$\check{D}^0_{\xx}(M)=  \check{H}^0(U,\tilde{M})$. 
For our purposes here  let $H_i(\xx^{(n)};M), i \in \mathbb{Z}$, 
denote the $i$-th Koszul homology of $M$ with respect to $\xx^{(n)} = x_1 ^n,\ldots,x_k^n$ 
for an integer $n \geq 1$. We prove the following:

\begin{theorem} \label{thm-i}
	Let $J = (x_1,\ldots,x_k)R$ denote a finitely generated ideal in a commutative ring $R$.
	For an $R$-module $M$ there is a commutative diagram 
	\[
	\xymatrix{
		 & 	\mathcal{D}_J(M) \ar[dl]_{\theta_M} \ar[dr]^{\rho_M} & \\
			\varprojlim_{x \in J}M_x \ar[rr]^-{\sigma_M}  &  &  \hole  \check{D}^0_{\xx}(M)
	}
	\]
	with the following properties:
	\begin{itemize}
		\item[(1)] $\theta_M$ and $\rho_M$ are injective and $\sigma_M$ is an isomorphism.
		\item[(2)] $\theta_M$ resp. $\rho_M$ is an isomorphism if and only if 
		$\check{H}^1_{\xx}(\mathcal{D}_J(M)) = 0$ (see \ref{not-1} (B) for $\check{H}^1_{\xx}(\cdot)$).
		\item[(3)] $\theta_M$ resp. $\rho_M$ is an isomorphism for every $M$ 
		if and only if the inverse system $\{H_1(\xx^{(n)};R)\}_{n \geq 1}$ is pro-zero, i.e. 
		for any $n$ there is an $m \geq n$ such that 
		$H_1(\xx^{(m)};R) \to H_1(\xx^{(n)};R)$ is zero. 
	\end{itemize}
\end{theorem}

Note that, if $M$ is a Noetherian module, then $\{H_i(\xx^{(n)};M)\}_{n \geq 1}$ is pro-zero 
for any system of elements $\xx$ and $i > 0$ as easily seen (see also \cite{Sp2}). 
Note that $\rho_M$ is in general not onto, 
even for an injective $R$-module $M$ (see Example \ref{ex-1}). The study of inverse systems 
that are pro-zero was initiated by Greenlees and May (see \cite{GM}) and Alonso Tarr\'io, Jerem\'ias L\'opez, Lipman (see 
\cite{ALL1}). For a further discussion about their notation we refer to \cite{Sp12} and to \cite{SS}. 
See also \ref{def-1} for the notion of weakly pro-regular sequences and related subjects. Moreover,  we present a description of $\check{D}^0_{\xx}(M)$ as the sheaf $\tilde{M}(U)$ (see Section 5 
for the details).

In Section 2 we recall some results about 
\v{Cech} and Koszul complexes needed in order to describe some properties of 
sequences generalizing those of regular sequences. In Section 3 we derive 
the homomorphisms of the above diagram. To this end we recall some constructions 
known in the case of Noetherian rings and modify them in the general case, 
not available in this form before. This is necessary 
for the proof of the main results done in Section 4. We continue in  Section 4 with a necessary 
and sufficient homological condition for $\rho_M$ to become an isomorphism and an 
Example clarifying the necessary conditions in \ref{thm-i}. In our notation we follow (with some minor 
differences) those of \cite{SS}. Moreover \cite{SS} is our basic reference. 

\section{Recalls about Sequences}
At the beginning let us fix some notation.  In the following we shall 
use these notations without further reference. Let $R$ denote a commutative 
ring and let $M$ denote an $R$-module. For a sequence of elements $\xx 
= x_1, \ldots, x_k$ and an integer $ n > 0$ put $\xx^{(n)} = x_1,^n\ldots,x_k^n$. 
Moreover, let $J = (x_1,\ldots,x_k)R$ the ideal generated by the sequence 
$\xx$. 

\begin{notation} \label{not-1}
	(A) We denote by $K_{\cdot}(\xx^{(n)};M)$ the 
	Koszul complex  of $M$ with respect 
	to the sequence $\xx^{(n)}$. For an integer $i$ let 
	$H_i(\xx^{(n)};M)$  denote the $i$-th 
	Koszul homology.  
	For two positive integers $m \geq n$ there are natural maps 
	$K_{\cdot}(\xx^{(m)};M) \to K_{\cdot}(\xx^{(n)};M)$ that induces homomorphisms on the corresponding homology modules. For each integer $i$ they induce inverse systems 
	$\{H_i(\xx^{(n)};M)\}_{n \geq1}$.\\
	(B) Let $K^{\cdot}(\xx^{(n)};M)$ denote the Koszul co-complex and $H^i(\xx^{(n)};M)$ its cohomology. It is known that $\varinjlim K^{\cdot}(\xx^{(n)};M) \cong \check{C}_{\xx}(M)$,  the \v{C}ech 
	complex with respect to $\xx$ (see e.g. \cite{SS}). That is, 
	\[
		\Cech(M) : 
		0\to M \stackrel{d^0}{\longrightarrow} \oplus_{i=1}^r M_{x_i} 
		\stackrel{d^1}{\longrightarrow} \oplus_{1 \leq i < j \leq k} M_{x_ix_j}
		\stackrel{d^2}{\longrightarrow} \ldots \to M_{x_1\cdots x_k} \to 0.
	\]
	We call $\check{H}^i_{\xx}(M) = H^i(\check{C}_{\xx}(M)), i \in 
	\mathbb{Z}$, the \v{C}ech cohomology of $M$ with respect to $\xx$.
	Moreover, let $\check{D}_{\xx}(M)$ denote 
	the global \v{C}ech complex given by 
	\[
	\check{D}_{\xx} (M) : 0 \to 
	\oplus_{i=1}^r M_{x_i} 
	\stackrel{d^1}{\longrightarrow} \oplus_{1 \leq i < j \leq k} M_{x_ix_j}
	\stackrel{d^2}{\longrightarrow} \ldots \to M_{x_1\cdots x_k} \to 0
	\]
	(see also \cite[6.1]{SS}). There is a short exact sequence 
	$0 \to \check{D}_{\xx} (M)[-1] \to \Cech(M) \to M \to 0$ of complexes that induces an exact 
	sequence 
	$0 \to \check{H}^0_{\xx}(M) \to M \to \check{D}^0_{\xx}(M) \to \check{H}^1_{\xx}(M) \to 0$ 
	of $R$-modules,
	where we abbreviate $\check{D}^0_{\xx}(M) := H^0( \check{D}_{\xx} (M))$. \\
	(C) Moreover let $\mathcal{D}_J(M) = \varinjlim \Hom_R(J^n,M)$ be  the ideal transform of 
	$M$ with respect to $J$. We refer to \cite{BrS}, \cite[Chaper 12, Section 5]{SS} and to 
	\cite{SS1} for more details.  There is a natural homomorphism $\tau_M: M\to \mathcal{D}_J(M)$ and a short exact sequence 
	\[
	0 \to \Gamma_J(M) \to M \stackrel{\tau_M}{\longrightarrow} 
	\mathcal{D}_J(M) \to H^1_J(M) \to 0, 
	\]
	where $\Gamma_J(\cdot)$ denotes the $J$-torsion functor and $H^i_J(\cdot)$ its right derived 
	functors, the local cohomology functors
\end{notation}

In the next we shall summarize some properties of sequences $\xx = x_1,\ldots,x_k$ partially needed 
in the sequel. 

\begin{definition} \label{def-1}
	(A) A sequence $\xx = x_1,\ldots,x_k$ 
	is called $M$-weakly pro-regular, if for all $i > 0$ the inverse system $\{H_i(\xx^{(n)};M)\}_{n \geq 0}$ 
	is pro-zero, where $H_i(\xx^{(m)};M) \to H_i(\xx^{(n)};M), m \geq n,$ denotes the natural 
	map induced by the Koszul complexes.  That is, for each integer $n$ there is an 
	integer $m \geq n$ such that the map $H_i(\xx^m;M) \to H_i(\xx^n;M)$ is zero. We 
	call $\xx$ weakly pro-regular if it is $R$-weakly pro-regular. 
	The first  study of weakly pro-regular sequences has been done in \cite{Sp2}. \\
	(B) In regard to \cite[Sect. 9, 6, def. 2]{Bn2} we call a sequence $\xx$ weakly secant 
	if the inverse system $\{H_1(\xx^{(n)};R)\}_{n \geq 1}$ is pro-zero. Therefore, a weakly 
	pro-regular sequence is weakly secant too. The converse does not hold. Moreover, if $R$ 
	is a Noetherian ring, then any sequence $\xx$ is weakly pro-regular. 
\end{definition}

In the following we give a characterization when $\xx$ is  weakly secant. This follows by 
some adaptions of arguments given in \cite[7.3.3]{SS} and \cite[Proposition 5.3]{Sp14}. 
For the ring $R$ let $R[T]$ denote the polynomial module in one variable. It is a free $R$-module with basis $\mathbb{N}$.

\begin{lemma} \label{lem-1}
	Let $\xx= x_1,\ldots,x_k$ denote a system of elements in $R$. Then the following conditions are equivalent:
	\begin{itemize}
		\item[(i)] $\xx$ is weakly secant.
		\item[(ii)] $\{H_1(\xx^{(n)}; F)\}_{n \geq 1}$ is pro-zero for any flat $R$-module $F$.
		\item[(iii)] $\varinjlim H^1(\xx^{(n)};I) = 0$ for any injective $R$-module $I$.
		\item[(iv)] $\check{H}^1_{\xx}(I) = 0$ for any injective $R$-module $I$. 
		\item[(v)] $\varprojlim H_1(\xx^{(n)};R[T]) = \varprojlim^1 H_1(\xx^{(n)};R[T]) = 0$. 
	\end{itemize}
\end{lemma}

\begin{proof}
	The equivalence of the first four conditions is a particular case of \cite[7.3.3]{SS} or 
	\cite[Lemma 2.4]{Sp2} for $i = 1$, 
	where weakly pro-regular sequences are characterized.  Next we show (ii) $\Longrightarrow$ 
	(v).  For $F = R[T]$ we get that $\{H_1(\xx^{(n)};R[T])\}_{n \geq 1}$ is pro-zero and (v) follows 
	(see e.g. \cite[1.2.4]{SS}). The implication (v) $\Longrightarrow$ (i) is true by \cite{Ei}  
	(see also the argument in the  paper \cite{Sp14}).
\end{proof}

\section{The Homomorphisms}
In the following we recall the homomorphisms of the diagram in the Introduction. Moreover we add  
a few comments. To this end we use the previous notation. 

\begin{proposition} \label{prop-1}
	($\rho_M$) We define 
	\[
	\rho_M : \mathcal{D}_J(M) \to \check{D}^0_{\xx}(M), \; \phi \mapsto (\phi_n(x_i^n)/x_i^n)_{i=1}^k,
	\]
	where $\phi_n \in \Hom_R(J^n,M)$ is a representative of $\phi \in \varinjlim \Hom_R(J^n,M)$. 
	Then $\rho_M$ is injective. 
\end{proposition} 

\begin{proof}
	It is immediate to see that $\rho_M(\phi)$ does not depend upon the representative $\phi_n$. 
	Moreover we have $\phi_n(x_i^n)/x_i^n = \phi_n(x_j^n)/x_j^n$ for all $i,j$. Therefore 
	$\rho_M(\phi) \in \Ker d^1 = \check{D}^0_{\xx}(M).$ Now let $\rho_M(\phi)  = 0$ and therefore 
	$\phi_n(x_i^n)/x_i^n =0$ for all $i = 1,\ldots,k$. That is $x_i^m \phi_n(x^n_i) = 0$ for some $m \geq 0$. 
	Because of $J^{n+m+k} \subseteq \xx^{n+m}R$ it follows that $\phi_n(J^{n+m+k}) = 0$. 
	Whence $\phi_n |_{J^{n+m+k}} = 0$ and therefore $\phi = 0$.  
\end{proof}

\begin{construction} \label{xyz}
Let $J \subset R$ denote an ideal. For an $R$-module $M$  we look at the system of localizations 
$\{M_x\}_{x \in J}$. For elements $x,y \in J$ we define a partial order $y \geq x$ whenever 
$x \in \Rad yR$, i.e., $x^k = y r$ for some $r \in R$. Then we define a homomorphism $\alpha_{x,y} : M_y\to M_x$ by 
\[
\alpha_{x,y} : M_y \to M_x, \; m/y^n \mapsto r^nm/x^{kn}.
\]
This is well-defined as easily seen. Moreover, if $x \in \Rad yR$ and $y \in \Rad zR$, 
then  it follows that 
$\alpha_{x,y} \cdot \alpha_{y,z} = \alpha_{x,z}$ and $\alpha_{x,x}  = \id_{M_x}$. So that 
 $\{M_x\}_{x \in J}$ with the homomorphisms $\alpha_{x,y}$ 
forms an inverse system.  Moreover, let 
$\alpha_x : \varprojlim_{x \in J} M_x  \to M_x$ denote the canonical map.
Note that for fractions $m/x_i^a$ and $n/x_j^{b}$ we use often the same exponent 
for the denominator. 
\end{construction}

\begin{proposition} \label{prop-2}
	($\theta_M$) We define 
	\[
	\theta_M: \mathcal{D}_J(M) \to \varprojlim_{x \in J} M_x, \; \phi \mapsto (\phi_n(x^n)/x^n)_{x\in J},
	\]
	where $\phi_n \in \Hom_R(J^n,M)$ is a representative of $\phi \in \varinjlim \Hom_R(J^n,M)$.
	Then $\theta_M$ is injective. 
\end{proposition}

\begin{proof}
	For any $x \in J$ and $\phi \in \mathcal{D}_J(M)$ let  $\phi_n \in \Hom_R(J^n,M)$ be a representative  of $\phi$. Then we define $\phi \mapsto \phi_n(x^n)/x^n \in M_x$,
	 which is well defined.  This is compatible with the map $\alpha_{x,y} : M_y \to M_x$ with 
	 $x \in \Rad yR$, say $x^k = yr$ as easily seen. By the universal property of inverse limits 
	 there is a homomorphism $\theta_M: \mathcal{D}_J(M) \to \varprojlim_{x \in J}M_x$. 
	 The injectivity of $\theta_M$ follows as in the proof of \ref{prop-1}.
\end{proof}

Now we are going on to construct the final morphism in the diagram of \ref{thm-i}. 

\begin{proposition} \label{prop-3}
	($\sigma_M$) By the above notation let $\alpha_x : \varprojlim_{x \in J} M_x \to M_x$ be the canonical map. We define 
	\[
	\sigma_M : \varprojlim_{x \in J} M_x \to \check{D}^0_{\xx}(M), \; f \mapsto (m_i/x_i^n)_{i=0}^k,  
	\text{ where } \alpha_{x_i}(f) = m_i/x_i^n, \; i = 1,\ldots,k.
	\]
	Moreover $\sigma_M$ 
	is an isomorphism and $\rho_M = \sigma_M \circ \theta_M$.
\end{proposition}

\begin{proof}
	First note that $\alpha_{x_ix_j,x_i}(\alpha_{x_i}(f)) = \alpha_{x_ix_j,x_j}(\alpha_{x_j}(f)) = \alpha_{x_ix_j}(f)$ for all pairs $i,j$. 
	This yields that $m_i/x_i^n = (x_j^nm_i)/(x_ix_j)^n = (x_i^nm_j)/(x_ix_j)^n = m_j/x_j^n$ and 
	$\sigma_M(f) \in  \check{D}^0_{\xx}(M)$. Let $\sigma_M(f) = 0$  for some $f \in \varprojlim  M_x$ and therefore 
	$\alpha_{x_i}(f) = m_i/x_i^n = 0$ for  $i = 1,\ldots,k$. Let $x \in J$ and $\alpha_x(f) = m/x^n$, 
	where $n$ can be chosen independently of $i,j$ and $x$.  Since $x^nm_i/(xx_i)^n =
	 x_i^nm/(xx_i)^n$ and $m_i/x_i^n = 0$, it follows that $\alpha_x(f) = m/x^n =0$ for all $x \in J$ 
	and $f = 0$, so $\sigma_M$ is injective.  
	
	Now let $(m_i/x_i^n)_{i=1}^k \in \check{D}^0_{\xx} (M)
	 = \Ker d^1$ and therefore $m_i/x_i^n = m_j/x_j^n$ for all $i,j$. That is, $(x_ix_j)^c x_j^n m_i= (x_ix_j)^c x_i^n m_j$ for some integer $c$ and $x_j^{c+n}m_i' = x_i^{c+n}m_j'$ with $m_l' =
	  x_l^cm_l$ and $m_l/x^n_l = m_l'/x_l^{c+n}, l = 1,\ldots,k$. Now choose an element  $y \in J$ and an integer $d$ such that $y^d = \sum_{i=1}^{k} r_ix_i^{c+n}$ since $y \in \Rad (x_1^{c+n}, \ldots, x_k^{c+n})$. We define 
	$m_y = \sum_{j=1}^{k} r_jm_j'$. Then 
	\[
	x_i^{c+n} m_y =  \small{\sum}_{j=1}^{k} x_i^{c+n}r_j m_j' =  \small{\sum}_{j=1}^{k} r_jx_j^{c+n}m_i' 
	= y^d m_i'
	\]
	and therefore $m_y/y^d = m_i/x_i^n$ for $i= 1,\ldots,k$. In order to show that  $(m_y/y^d)_{y \in J}$ defines an element $f \in \varprojlim_{x \in J} M_x$ such that $\sigma_M(f) = (m_i/x_i^n)_{i=1}^k$ let 
	$z \in J$  with $z^f = yr$, where $m_z/z^e$ is chosen as before. By the equation  above  
	$x_i^{c+n}r^d m_y = z^{fd}m_i'$ and therefore 
	$r^dm_y/z^{df} = m_i/x_i^n = m_z/z^e$ for $i = 1,\ldots,k$. That is, $\alpha_{z,y}(m_y/y^d) =
	 m_z/z^e$,  whence $f$ is an element of the inverse limit. The final claim is obvious.
\end{proof}

\begin{question} \label{question}
	Fix the previous notation. What is a necessary and sufficient condition on the sequence 
	$\xx = x_1,\ldots,x_k$ such that the homomorphisms 
	$$
	\rho_M: \mathcal{D}_J(M) \to \check{D}^0_{\xx}(M) \text{ and }
	\theta_M: \mathcal{D}_J(M) \to \varprojlim{}_{x \in J} M_x, \; \phi \mapsto (\phi_n(x^n)/x^n)_{x \in J}
	$$
	in \ref{prop-1} and in  \ref{prop-2} become  isomorphisms for an $R$-module $M$? 
	For an answer see \ref{prop-4} below.
\end{question}

\section{Proof of the Main Results}
In the following we shall discuss when the $R$-homomorphism $\rho_M : \mathcal{D}_J(M) \to \check{D}^0_{\xx}(M)$ is onto. To this end we need a technical result. 

\begin{lemma} \label{lem-2}	
	Let $\xx$ and $J = {\xx}R$ be as above. Let $M$ denote an $R$-module.
		 There is a commutative diagram with exact rows
		\[
		\xymatrix{
			0 \ar[r] & \Gamma_J(M) \ar[r]  \ar@2{-}[d] & M \ar[r] \ar@2{-}[d] & \mathcal{D}_J(M) 
			\ar[r] \ar@{^{>}->}[d]^{\rho_M} & H^1_J(M )\ar[r] \ar@{^{>}->}[d] & 0 \\
			0 \ar[r] &   \check{H}^0_{\xx}(M) \ar[r]  & M \ar[r]  &\check{D}^0_{\xx}(M)
			\ar[r] &  \check{H}^1_{\xx}(M)\ar[r]  & 0.
		}
		\]
		The natural map $H^1_J(M) \to \check{H}^1_{\xx}(M)$ is injective and $ \check{D}^0_{\xx}(M)/\mathcal{D}_J(M) \cong 
	 	\check{H}^1_{\xx}(M)/H^1_J(M)$. Moreover,  $\rho_M$ is an isomorphism if $\xx$ is weakly 
	 	pro-regular. 
\end{lemma}

\begin{proof} 
	For the exact sequence at the bottom of the diagram see \ref{not-1} (B). The  exact sequence at the top  is shown in \ref{not-1} (C). The commutativity of the diagram is 
	obvious. Because the third vertical  map $ \mathcal{D}_J(M) \to 
	\check{D}^0_{\xx}(M)$ is injective (see \ref{prop-1}), so is the 
	fourth one. The statement  follows now. Finally if $\xx$ is weakly pro-regular, 
	then $H^1_J(M) \cong \check{H}^1_{\xx}(M)$ (see e.g. \cite[7.4.5]{SS}). 
\end{proof}

It could be of some interest to describe a necessary and sufficient condition for $\rho_M$ to 
become an isomorphism in terms of $\xx$ and the $R$-module $M$. 

\begin{proposition} \label{prop-4}
	With the previous notation there is a short exact sequence 
	\[
	0  \to \mathcal{D}_J(M) \stackrel{\rho_M}{\longrightarrow} \check{D}^0_{\xx}(M) 
	\to \check{H}^1_{\xx}(\mathcal{D}_J(M)) \to 0
	\]
	and an isomorphism $\check{H}^1_{\xx}(M)/H^1_J(M) \cong \check{H}^1_{\xx}(\mathcal{D}_J(M)) $.
	Moreover $\rho_M$ is an isomorphism if and only if $\check{H}^1_{\xx}(\mathcal{D}_J(M))= 0$. 
\end{proposition}

\begin{proof}
	Localizing the exact sequence of \ref{not-1} (C) at $x \in J$ yields $M_x \cong 
	 \mathcal{D}_J(M)_x $ and therefore 
	\[
	\check{D}^0_{\xx}(\mathcal{D}_J(M)) \cong \varprojlim{}_{x \in J} \mathcal{D}_J(M)_x 
	\cong \varprojlim{}_{x \in J}M_x \cong \check{D}^0_{\xx}(M)
	\] 
	(see \ref{prop-3}).
	By the exact sequence in \ref{not-1} (B) for the ideal transform $\mathcal{D}_J(M)$ it 
	yields the short exact sequence of the statement. 
\end{proof}

\noindent \textbf{Proof of Theorem \ref{thm-i}:} \\
\noindent (1): These statements are shown in Propositions \ref{prop-1}, \ref{prop-2} and \ref{prop-3}. \\
(2): We have that $\rho_M$ is an isomorphism if and only if $\theta_M$ is so. By \ref{prop-4} $\rho_M$ 
is an isomorphism if and only if $\check{H}^1_{\xx}(\mathcal{D}_J(M))= 0$. \\
(3):  If $\rho_M$ is an isomorphisms for any $R$-module it holds in particular for any injective 
$R$-module. In the diagram of \ref{lem-2} we have the vanishing  $H^1_J(I) = 0$ for any injective $R$-module $I$. 
That is, $\rho_I : \mathcal{D}_J(I) \to \check{D}^0_{\xx}(I)$ is an isomorphism for every 
injective $R$-module $I$ if and only if $\check{H}^1_{\xx}(I) = 0$ for every injective $R$-module $I$. 
By \ref{lem-1} this holds if and only if $\xx$ is a weakly secant sequence. 
Now we prove that $\rho_M$ is an isomorphism too. Let 
$0 \to M \to I^0 \to I^1$ be the beginning part of an injective resolution of $M$. It induces a commutative 
diagram with exact rows 
	\[
	\xymatrix{
	0 \ar[r] & \mathcal{D}_J(M) \ar[r]  \ar[d]^{\theta_M} & \mathcal{D}_J(I^0)  \ar[r] \ar[d]^{\theta_{I^0}}& 
	\mathcal{D}_J(I^1)   \ar[d]^{\theta_{I^1}}\\
	0 \ar[r] &  \varprojlim_{x \in J} M_x  \ar[r]  & \varprojlim_{x \in J} (I^0)_x \ar[r]  &\varprojlim_{x \in J}  (I^1)_x. 
	}
	\]
For the first exact sequence recall $\mathcal{D}_J(M) = \ker(\mathcal{D}_J(I^0)  \to \mathcal{D}_J(I^1))$. 
For the second note that localization is exact and passing to the inverse limit
left exact. If $\xx$ is weakly secant, then $\theta_{I^i}, i = 0,1,$ is an isomorphism and $\theta_M$ too. 
 \hfill $\square$
\medskip

We conclude with an explicit example of a ring and an  injective module such that $\rho$ 
is not an isomorphism.

\begin{example} \label{ex-1} (see also \cite[Example 5.5]{SS1})
	Let $R = \Bbbk[[x]]$ denote the power series ring in one variable over 
	the field $\Bbbk$. Let $E = E_R(\Bbbk)$ denote the injective hull 
	of the residue field. Then define $S = R \ltimes E$, the idealization  
	of $R$ by the $R$-module $E$ as introduced by M. Nagata (see \cite{Nm}). That is, $S=R\oplus E$ as an 
	$R$-module with a multiplication on $S$  defined by $(r, r)\cdot (r', e')=(rr', re'+r'e)$ for all $r, r' \in R$ and $e, e'\in E$.	
	By a result of Faith \cite{Fc} we have that the commutative ring $S$ is self-injective. More precisely, 
	there is an isomorphism of $S$-modules $\Hom_R(S, E)\cong S$  
	(see also  \cite[Theorem A.4.6]{SS}). We consider the ideal $J :=(x, 0)S$ 
	of $S$ and note $\Gamma_J(S) =0\ltimes E$. Then 
	$\Gamma_J(S)$ is not injective as an $S$-module (see \cite[2.8.8]{SS}).
	We have that $S$ is self-injective with 
	$\Gamma_J (S) = 0 \ltimes E$. Hence $\mathcal{D}_J(S) =R \subsetneqq
	\check{D}^0_{(x,0)}(S) =S_{(x, 0)}=R_x$. 
	
	In this example the ascending sequence of ideals $0:_S(x, 0)^t = 0 \ltimes (0:_E x^t), t>0$
	does not stabilizes. That is, $S$ is not of bounded $(x,0)$-torsion and therefore the inverse 
	system $\{H_1((x,0)^n;S)\}_{n \geq 1}$ is not pro-zero.
\end{example}

\section{The associated sheaf}

In this part of the paper we shall recall the sheaf construction 
of $\tilde{M}(U)$ for a finitely generated ideal $J = (x_1,\ldots,x_k)R$ 
in a commutative ring $R$ and an $R$-module $M$. Note that it is closely related 
to the cohomological investigations.  Set $X = \Spec R$ and $U = 
X \setminus V(J)$. Note that $D(x_i) = X \setminus V(x_i), i = 1,\ldots,k,$ is an open 
covering of $U$. 

\begin{defob} \label{app-1}
	For the $R$-module $M$ we consider the set $\{M_x\}_{x\in J}$. Because we are 
	interested in the localization $M_x, x \in J,$ we may replace $M$ by $M/0:_M \la J \ra$, where 
	$0:_M \la J \ra = \cup_{n \geq 1} 0:_M J^n$. Note that $M_x = M_y$ for $x,y \in J$ with 
	$\Rad xR = \Rad yR$. For elements $x,y \in J$ with 
	$x \in \Rad yR$, i.e., $x^k = y r$ for some $r \in R$ we have as above  a homomorphism $\alpha_{x,y} : M_y\to M_x$ (see \ref{xyz} ). 
	Now recall that  $\{D(x_i) | i = 1,\ldots,k\},$ is an affine 
	covering of $U$.
	
	\textit{Claim:}
	 $\tilde{M}(U):= (\{M_x\}_{x \in J}, \{\alpha_{x,y}| \Rad xR 
	\subseteq \Rad yR\}) = (\{M_x\}_{x \in J}, \{\alpha_{x,y}| D(x)
	\subseteq D(y)\})$ is a sheaf of modules. 
	
	\textit{Proof:}
	To this end we have to show the following:
	\begin{enumerate}
		\item If $m/x^n \in M_x$ maps to zero in $R_{x_i}$ for all $i = 1,\ldots,k,$ then $m=0$.
		\item  If $m_j/x_j^n\in M_{x_j} , j=1,\ldots,k,$ satisfies $m_i/x_i^n = m_j/x_j^n$ in $M_{x_ix_j}$
		for all $i,j$.
		Then there is an $n/y \in M_y$  that maps to $ m_i/x_i^n$ for all $i = 1,\ldots,k$.
	\end{enumerate}
	If $m/x^n$ maps to zero in $R_{x_i}$, then $x_i^c m= 0$ for all $i = 1,\ldots,k$, 
	where $c$ can be chosen independently of $i$. That is $J^{ck} m= 0$ and $m=0$ and (1) holds. 
	In order to show (2) note that we may 
	chose $n$ independently of $i,j$. Then the  assumption implies $(x_ix_j)^cx_j^nm_i = (x_ix_j)^cx_i^nm_j$ for a certain $c$ for all $i,j$. We put $m_i' = x_i^cm_i$ and get 
	$m_i/x_i^n = m_i'/x_i^{c+n}$ and $x_j^{c+n}m_i' = x_i^{c+n} m_j'$ for all $i,j$. Now we choose 
	$y = \sum_{i=1}^{k}r_ix_i^{c+n}$ and $m = \sum_{j=1}^{k} r_j m_j'$. Then 
	\[
	x_i^{c+n} m =  \small{\sum}_{j=1}^{k} r_j x_i^{c+n}m_j' =  \small{\sum}_{j=1}^{k}  r_j x_j^{c+n} m_i' 
	= y m_i'
	\]
	and $m/y = m_i'/x_i^{c+n} = m_i/x_i^n$ for all $i = 1,\ldots,k$. 
\end{defob}

Moreover, clearly $\tilde{M}(U)$ is a sheaf of $\tilde{A}(U)$-modules. 
We conclude with the obvious remark that $\tilde{M}(U)$ coincides with 
$\check{D}^0_{\xx}(M)$  for a finitely generated ideal.

\begin{corollary} \label{cor-1}
	For an $R$-module $M$ and an ideal $J \subset R$ we have $\check{D}^0_{\xx}(M) 
	\cong \tilde{M}(U)$, where $\xx =x_1,\ldots,x_k, J = (x_1,\ldots,x_k)R$ and $U = X \setminus V(J)$.
\end{corollary}

\begin{proof} 
	Let $f = (m_i/x_i^n)_{i=1}^k \in \check{D}^0_{\xx}(M)$ and therefore $m_i/x_i^n = m_j /x_j^n$ 
	for all $i,j$. By virtue of (2) in  \ref{app-1} there is an $n/y \in M_y$ that maps to $m_i/x_i^n$ 
	for all $i = 1,\ldots,k$. Whenever, $n/y$ maps to zero by the map  of \ref{app-1}, the  condition 
	(1) in \ref{app-1} implies that $n = 0$. So, there is an isomorphism  $\check{D}^0_{\xx}(M) 
	\cong \tilde{M}(U)$. 
\end{proof}

	{\bfseries{Acknowledgement.}} Thanks to the reviewer for the careful reading of the manuscript 
and the suggestions. 

\bibliographystyle{siam}

\bibliography{hart-1}

\end{document}